\documentclass[12pt]{amsart}
\usepackage{amsfonts}
\usepackage{amsfonts,latexsym,rawfonts,amsmath,amssymb,amsthm}
\usepackage[plainpages=false]{hyperref}

\usepackage{graphicx}

\RequirePackage{color}

\numberwithin{equation}{section}

\newcommand{\beq}{\begin{equation}}
\newcommand{\eeq}{\end{equation}}
\newcommand{\beqs}{\begin{eqnarray*}}
\newcommand{\eeqs}{\end{eqnarray*}}
\newcommand{\beqn}{\begin{eqnarray}}
\newcommand{\eeqn}{\end{eqnarray}}
\newcommand{\beqa}{\begin{array}}
\newcommand{\eeqa}{\end{array}}

\def\p{\partial }

\def\Om{\Omega}
\def\pom{\p  \Omega}
\def\bom{\overline  \Omega}

\newtheorem{Proposition}{Proposition}[section]
\newtheorem{Theorem}[Proposition]{Theorem}
\newtheorem{Lemma}[Proposition]{Lemma}

\newtheorem{Corollary}[Proposition]{Corollary}

\title  {Sharp boundary regularity for some degenerate-singular Monge-Amp\`ere   Equations  on $k$-convex domain}

\begin{document}

\address{Huaiyu Jian: Department of Mathematical sciences, Tsinghua University, Beijing 100084, China.}

\address{Xianduo Wang: Department of Mathematical sciences, Tsinghua University, Beijing 100084, China.}

\email{
hjian@tsinghua.edu.cn;  xd-wang18@mails.tsinghua.edu.cn}

\thanks{This work was supported by NSFC 12141103  }


\bibliographystyle{plain}

\maketitle

\baselineskip=15.8pt
\parskip=3pt

\centerline {\bf   Huaiyu Jian \ \ \ \ Xianduo Wang}

\centerline {   Department of Mathematical sciences, Tsinghua University, Beijing 100084, China}

\vskip20pt

\noindent {\bf Abstract}:
We introduce the concept of $k$-strictly  convexity to describe the accurate convexity of  convex domains some directions of  which boundary may be flat. Basing this accurate convexity we construct  sub-solutions the Dirichlet problem for some   degenerate-singular Monge-Amp\`ere type equations
and  prove the sharp  boundary estimates for  convex viscosity solutions of the problem.  As a result,
 we obtain the optimal global H\"older regularity of the convex viscosity  solutions.


 \vskip20pt


\vskip20pt

\noindent {\bf  Running head:} Sharp regularity for    Monge-Amp\`ere   equation

\vskip20pt

\baselineskip=15.8pt
\parskip=3pt

\newpage

\centerline {\bf  Sharp boundary regularity for some }
\centerline {\bf degenerate-singular Monge-Amp\`ere   Equations  on $k$-convex domain  }

 \vskip10pt

\centerline {\bf   Huaiyu Jian \ \ \ \ Xianduo Wang }

\maketitle

\baselineskip=15.8pt
\parskip=3.0pt

\section {Introduction}
Affine hyperbolic sphere   is  a well-known important model in affine geometry \cite {[Ca2], [CY], [CY2], [JW]} as well as a fundamental model in
  affine sphere relativity \cite{[M]}. It is determined by the Legendre transform of the solution to
the  following Drichlet problem of Monge-Amp\`ere equation:
 \beq
 \begin{split} \label {1.1}
\det  D^2 u&=  \frac{1}{|u|^{n+2}}\ \  \text{in}\ \Om,\\
   u&=0\ \ \text{on}\ \pom,
\end{split}
\eeq
where $\Om$ is a convex domain in $R^n$ $(n\geq 2).$   Also, if $u$ is a solution to problem (1.1),
  $(-u)^{-1}\sum u_{x_ix_j} dx_idx_j$   gives the Hilbert metric (Poincare metric) in the convex domain $\Om$ \cite {[LN]}.

  When $\Om$ is a bounded uniformly convex $C^2$-domain, the existence and uniqueness of
solutions  to \eqref{1.1}  in the space  $C^\infty(\Om)\cap C(\bom)$
was obtained in \cite {[CY]}.  In fact, the same result was proved in  \cite {[CY]} for general equation
$\det  D^2 u=  f(x, u) $ for  general $f$ satisfying the following (1.3) and (1.4) for $\alpha>0$ and $\gamma =0$.
If $\Om$ is an unbounded convex domain such that  $\pom$ is strictly convex at some point $x_{0}\in \pom$,
then  problem (1.1) admits  a convex solution $u\in C^{\infty}(\Om)\bigcap C (\overline{\Om})$. See \cite {[JL1]} for the details.
Besides, the following three examples were verified in  \cite {[JL1]}:
\beqs
 \begin{split}
&(i)\;  \text{If}\;  \Om=B_{1}(0), \text{then the unique solution to (1.1) is } u(x)=-\sqrt{1-|x|^{2}};\\
&(ii)\; \text{ If }\; \Om=B_{1}^{n-1}(0)\times R^{+}, \; \text{then}\;
 u (x)=-(n+1)^{\frac{1}{2}} n^{-\frac{n}{2(n+1)}}x_{n}^{\frac{1}{n+1}}(1-|x'|^{2})^{\frac{n}{2(n+1)}}\\
 & \text { is a solution to (1.1)};\\
 &(iii)\; \text{If }\; \Om=\{(x',x_{n})\in R^{n}:x_{n}\geq \sqrt{\frac{n^{n}}{(n+1)^{n+1}}}|x'|\}, \ \  \text{then}\\
& \ \ \ u (x)=-[\frac {(n+1)^{n+1}}{n^n}( x_{n} )^{2}-|x'|^{2}]^{\frac{n}{2(n+1)}}\; \text{  is a solution to (1.1).}\\
 \end{split}
\eeqs
Here and below, we use notation   $ x=(x', x_n), \; x'=(x_1, \cdots, x_{n-1}).$
 Obviously, examples (i)-(iii) shows that the boundary behaviour of solutions to problem (1.1) depends on the convexity of the domains $\Om$.

 In this paper we first introduce the concept of {\sl $k$-strictly   convexity }
 to describe the convexity of the domains accurately.  Then we will prove that  the same boundary behaviour as
 examples (i)-(iii)
 happens to a class of  Monge-Amp\`ere type equations which may be singular and degenerate on the boundary.

 \noindent{\bf  Definition 1.1}. {\sl Suppose $\Om$ is a bounded convex open domain in $R^n$, $x_0\in \pom$, $1\le k\leq n-1$ is positive
 integers,  and $a_i\geq 1$  for $i=1, 2, \cdots , k$.
We say that $\Omega$ satisfies  $ k$-strictly  convex condition with $(a_{1},..., a_{k})$ at $x_{0}$   if
there exist  positive constants $ \eta_{1},..., \eta_{k} $, after suitable coordinate translation and rotation transforms, such that
$$x_{0}=\textbf{0}  \ \   \text{and} \ \   \Om\subseteq\{x=(x_1, \cdots, x_k, \cdots, x_n)\in R^{n}|x_{k+1}>\eta_{1}|x_{1}|^{a_{1}}+...+\eta_k|x_{k}|^{a_k}\}.$$
If $\Omega$ satisfies  $ k$-strictly  convex condition  with $(a_{1},..., a_{k})$ at every $x_{0}\in \pom$ with the uniformly positive constants $ \eta_{1},..., \eta_{k}$, we say that $\Omega $
is  $k$-strictly convex domain  with $(a_{1},..., a_{k})$. }

  The numbers $a_{1},..., a_k$     describe the convexity   at $x_{0}$   exactly along each directions.
   The less are $a_i$, the more
   convex is $\Omega$ along the $x_i$-direction.  If $k<n-1$, the domain is flat near $x_{0}$ along $x_{k+2}, \cdots, x_{n}$-directions,
   which may be viewed as $n-1$-strictly convex with  $a_{k+1}, \cdots, a_{n-1}=+\infty$. Note that since $ x_{0}=\textbf{0} $ and $\eta_i$ can be taken to be large enough,
   it is sufficient to require that $|(x_1, \cdots, x_{k})|< 1$ in the above Definition.  A domain part of which boundary is flat
   may be viewed either as $ 0$-strictly   convex domain, or $k$-strictly convex with  $a_{1}, \cdots, a_{k}=+\infty$.

 Denote   $d_{x}=dist(x, \pom)$, and   consider  the Dirichelet problem for the Monge-Amp\`ere type equation
\begin{equation}
\begin{split}\label{1.2}
\det  D^2 u&=  F(x, u, \nabla u)\ \  \text{in}\ \Om,\\
   u &=0\ \ \text{on}\ \pom,
\end{split}
 \end{equation}
where the known function $F\in C (\Om\times(-\infty,0)\times  R^n)$
satisfies
\begin{equation}\label{1.3}
 \text{    for any} \; (x,  q) \in \Om\times  R^n,
F(x, z, q) \text{ is
non-decreasing in}\ z \in (-\infty,0),
 \end{equation}  and  there are constants $\alpha, \beta ,  \gamma , A $ and $ a_i \; (1\leq i\leq k)$  such that
\begin{equation}\label{1.4}
\begin{split}
&  a_i\geq 1 \; (1\leq i\leq k),  \; A>0, \; \beta\geq (n+1), \;\\ 
&0<F(x, z, q)\leq A d_{x}^{\beta-(n+1)}|z|^{-\alpha}(1+q^{2})^{\frac{\gamma}{2}}, \ \ \ \forall\; (x, z, q) \in \Om\times(-\infty,0)\times R^n , \\
&  0< \bar a +\beta-n-\gamma+1<n+\alpha-\gamma , \ \  \bar a:=\sum_{i=1}^k\frac{2}{a_{i}} .
 \end{split}
\end{equation}

To utilizing the $k$-strictly   convexity  of the domain, we need additional assumption:
\begin{equation} \label{1.5}
  F(x, u, \nabla u)  \text{ is invariant under coordinate translation and rotation transforms}.
 \end{equation}
 Obviously, (1.5) is satisfied if $F(x, z, p)=f(x,z, |p|)$ for some function $f$ defined in $\Om\times(-\infty,0)\times[0, \infty)$.

The  main result of this paper is
	\begin{Theorem}\label{1.1}
Supposed that $\Om$  is a  bounded convex domain in $R^n$, $x_{0}\in \Om$, and $\Om$ satisfies  $k$-strictly  convex condition  with $(a_{1},..., a_{k})$ at $x_{0}$.
If $F$ satisfies (1.3)-(1.5) and  $u\in C(\overline \Om)$ is a  convex viscosity solution to problem (1.2), then there exists a positive constant $C$, depending only on  $a_{1},..., a_{k}, n, A, \alpha, \beta, \gamma$ and the
$\eta_1, \cdots, \eta_{k}$ in the $k$-strictly  convex condition at $x_{0}$,  such that
\begin{equation}\label{1.6}
|u( x)|\leq C (d_x)^{\mu} \ \ \text{for all}\;  x\in \Om \; \text{such that }\; d_x=|x-x_0|,
\end{equation} where
\begin{equation}\label{1.7}
 \mu=\frac{\bar a+\beta-n-\gamma+1}{n+\alpha-\gamma}.
 \end{equation}
 \end{Theorem}
We refer to Section 2 for the  viscosity solution of problem (1.2).
 The global H\"older regularity will follow directly from Theorem 1.1.

 \begin{Corollary} \label {1.2}
Supposed that $\Om \subset R^n$  is     $ k$-strictly   convex domain  with $(a_{1},..., a_{k})$.   If $F$ satisfies (1.3)-(1.5) and  $u\in C (\overline\Om)$ is a  convex viscosity solution to problem (1.2),
then $u\in C^{\mu}(\overline{\Omega})$ and  there exists a positive constant $C=C(\Om, n, A, \alpha, \beta, \gamma)$ such that $|u|_{C^{\mu}(\overline{\Omega})}\leq C$, where  $\mu$ is given again by (1.7).
 \end{Corollary}

 As we have said, a bounded convex domain (which boundary may contain a flat part) may be viewed as a  $ 1$-strictly exterior convex domain  with $a_{1}=\infty$.  Taking $k=1$ and letting $a_1\to \infty$ in (1.7), we
 see that the $\mu$ is turned to $\mu_0:= \frac{ \beta-n-\gamma+1}{n+\alpha-\gamma}.$  Then Corollary 1.2  should be
 \begin{Corollary} \label {1.3}
 Supposed that $\Om \subset R^n$  is       a bounded  convex domain.   If $F$ satisfies (1.3)-(1.5) and  $u\in C (\overline\Om)$ is a  convex viscosity solution to problem (1.2),
then $u\in C^{\mu_0}(\overline{\Omega})$ and  there exists a positive constant $C=C(\Om, n, A, \alpha, \beta, \gamma)$ such that $|u|_{C^{\mu_0}(\overline{\Omega})}\leq C$.
 \end{Corollary}

Notice that  examples (i)-(iii)  shows  the estimate  (1.6)  is sharp and so Corollary 1.2 and   1.3 are   optimal.

Besides the geometric models included in  problem (1.1),  problem (1.2) also includes the well-known  prescribed Gauss curvature problem
(when $F=\eta(x, u)(1+|Du|^{2})^{\frac{n+2}{2}}$). Moreover,
the projection of $H$ on the plane $\{x_{n+1}=-1\}$, where $H$
 is the solution to  general  $L_p$-Minkowski problem, satisfying
 $$\det (\nabla^2 H+ HI)=\frac{\eta(x, \nabla H)}{H^p},$$  also reduces to  problem (1.2).
  See \cite{ [CW], [JLW],[JLZ],[Lut]} for the details.

It should be pointed out that problem (1.2) may be degenerate, singular,
or both degenerate and singular on the boundary since $u =0$ on $\partial \Omega$.    There are general existence result for the solution in space $C^{2}(\Omega)\cap C(\bar \Omega)$. For example, see
\cite{[CY], [CNS],[TU]} for convex domains with smooth (or $C^{1,1}$) boundary. These results may be extended to the case of  any bounded convex
domain, due to the a prior estimate obtained in Corollary 1.2 and   1.3.  Also,  there are a few results of global  H\"older regularity for problem (1.2) with particular $F $ \cite{[JL], [JLT],[Le], [U1]}, and many important papers on
 global regularity better than H\"older continuity for    Monge-Amp\`ere
equations under more strong assumption on $F$ and $\Omega$ \cite{[CNS], [F], [G],[LS], [Sa], [TW]}. But all of the results mentioned above ignore  the influence of anisotropic convexity of the domain on the boundary behaviour of the solutions.

Recently, we   introduced  the concept of  $n-1$-strictly  convexity for the domain in \cite{[JLL]} and  obtained the results of Theorems 1.1 only for the case $k=n-1$.    However, those results
can not be used to explain the exact behaviour of example (ii) and (iii), since the cylindrical surface and conical surface are $n-2$-strictly   convex.

This paper is organized as follows. In Section 2, we  derive a comparison principle for the viscosity solution to problem (1.2).
In Section 3 we  construct  smooth subsolutions  to problem (1.2), which, together the comparison principle,   proves Theorems 1.1  and Corollary 1.2 and 1.3.   As we see,
the arguments constructing  subsolution will be  delicate and
  technical.

\section {comparison principle for  viscosity solutions}

For convenience, we state the definition for viscosity solution to problem (1.2) and derive two comparison principles, which
should be known for specialists. One can see the   paper \cite{[Cr]} and the references therein  on viscous solution theory for general elliptic equations.

{\bf Definition 2.1}\; Let $\Omega \subset \mathbb{R}^n$ be an open  set,  $u\in C(\overline{\Omega})$ and $ F(x, z, p) $ be a non-negative function defined in $\Omega\times R\times R^n$.
{\sl We say that   $u$ is a   viscosity subsolution (supersolution) of the equation
\begin{equation}\label{2.1}\det D^2u=F(x, u, Du)\;\;   in \;\; \Omega
 \end{equation} if whenever convex $\phi \in C^2(\Omega) $
and $x_0\in \Omega$ are such that $(u-\phi)(x)\leq (\geq ) (u-\phi)(x_0)$ for all $x$ in a neighborhood of $x_0$, then we must have
$$\det D^2\phi (x_0)\geq (\leq )F(x_0, u(x_0), D\phi(x_0)).$$ }   {\sl If $u$ is both viscosity subsolution and viscosity supersolution, then $u$ is called a viscosity solution.}

It is obvious that a convex $C^2$-solution  of (2.1)  must be  convex viscosity solution.  Caffarelli   indicated that the (Alexandrove) generalized solution (see \cite {[F], [G]}) to the   equation  $\det D^2u=\eta(x)$ is equivalent to the viscosity solution if $\eta\in C(\Omega)$.  Hence, a $C^1$-(Alexandrove) generalized solution to (2.1) is   convex viscosity solution if $F\in C(\Om \times R\times R^n)$.

\begin{Lemma} 
		Suppose that $\Omega$ is a bounded convex domain in $ {R}^n$, $F(x, z, p)$ is a nonnegative function  satisfying
\begin{equation}\label{2.1}
 \text{ for any} \; (x,  q) \in \Om\times  R^n,
F(x, z, q) \text { is non-decreasing in
 }\; z \in (-\infty, \infty),
\end{equation}
  and $u \in   C(\overline{ \Omega})$  is  a viscosity solution to (2.1).
		
(i)\;  If $ v \in C^2(\Omega)\cap C(\overline{ \Omega})$ is a convex function in $\Omega$,  satisfying
	 \begin{equation}\label{2.3}
		  \det D^2 v >  F(x, v, Dv) \;   \text{in}\;  \Omega,  \;  \text{  and}\; \; u \ge v \; \text{on}\;  \partial \Omega,
		 \end{equation}
	  then   $u \ge v$ on $\overline{  \Omega}$.

(ii)\; If $ v \in C^2(\Omega)\cap C(\overline{ \Omega})$ is a convex function in $\Omega$,  satisfying
		\begin{equation}\label{2.4}
		  \det D^2 v < F(x, v, Dv) \;  \text{in}\;  \Omega,  \;  \text{  and}\; u \leq v \; \text{on}\;  \partial \Omega,
		 \end{equation}
		  then   $u \leq v$ on $\overline{ \Omega}$.
	\end{Lemma}
	\begin{proof}
		Since the  proofs of (i) and (ii)  are similar, it is enough to prove (i).
		By the contradiction argument we assume that $\inf_{\Omega} (u - v) <0$, we use the fact that $u \ge v$ on $\partial \Omega$
to find   an  $x_0 \in \Omega$
such that $$u(x_0) - v(x_0) = \inf_{\Omega} (u - v)<0.$$   Since $v$ is convex, by the definition of viscosity solution we have
$$\det D^2 v(x_0) \le F(x_0, u(x_0), Dv(x_0)).$$
On the other hand, by the assumption (2.3) for $v$ and (2.2) for $F$, we have
		\[
		det D^2 v(x_0)  > F(x_0, v(x_0), Dv(x_0)) \geq  F(x_0, u(x_0), Dv(x_0)) ,
		\] which is impossible. So our assumption does not hold and it shows $u  \ge v$ on  $\overline{  \Omega}$.
	\end{proof}

Checking the above proof we see that if $F(x,z, p)$ is strictly increasing in $z$, the $>$ in (2.3) and the $<$ in (2.4)
 may be replaced by $\geq $ and $\leq $ respectively. In addition, we can exchange the equation
for $u$ and the inequalities for $v$. Consequently, we have

\begin{Lemma} 
		Suppose that  $\Omega$ is a bounded convex domain in $ {R}^n$, $F(x, z, p)$ is a nonnegative function  satisfying
\begin{equation}\label{2.5}
 \text{ for any} \; (x,  q) \in \Om\times  R^n,
F(x, z, q) \text { is strictly increasing in
 }\; z \in (-\infty, \infty),
\end{equation}
  and $u \in C(\Omega)\cap C(\overline{ \Omega})$  is a viscosity sub-solution to
		\[
		\det D^2 u \geq F(x, u, Du)  \text{in}\;  \Omega .
		\]
  If $ v \in C^2(\Omega)\cap C(\overline{ \Omega})$ is a convex function in $\Omega$,  satisfying
	 \begin{equation*}
		  \det D^2 v =  F(x, v, Dv) \;   \text{in}\;  \Omega, \; \text{and}\;  u \leq  v \; \text{on }\; \partial \Omega,
		 \end{equation*}
		 then   $u \leq v$ on $\overline{  \Omega}$.
\end{Lemma}

The following  Lemma was proved in \cite{[JL]}, which will be used to prove Corollary 1.2 and 1.3.

\begin{Lemma}\label{2.3}
Let $\Om$ be a bounded convex domain and $u\in C(\overline{\Om})$ be a  convex  function in $\Om$ with $u|_{\pom}=0$. If there are $\widetilde{\mu}\in(0,1]$ and  $C>0$
such that
 $$ |u(x)|\leq C{d_{x}}^{\widetilde{\mu}},\ \ \forall x\in \Om ,$$
  then $u\in C^{\widetilde{\mu}}(\overline{\Om})$ and
\begin{equation*}
 |u|_{C^{\widetilde{\mu}}(\overline{\Om})}\leq  C\{1+[diam(\Om)]^{\widetilde{\mu}}\}.
 \end{equation*}
\end{Lemma}

\section { Proof of theorem 1.1 and its corollaries}

In this  section we 
 first construct subsolutions to problem (1.1).  This construction is technical and delicate. Then we will use Lemmas 2.1 and 2.3 to
prove Theorems 1.1   and Corollary  1.2,  while Corollary 1.3 follows directly from the proof of Corollary 1.2.

{\bf Proof of Theorem 1.1.}\;\; We assume $k\leq n-2$, since the case $k=n-1$ was proved in \cite{[JLL]}.

{\bf Step 1.}\; {\sl  Normalize the domain and the conclusion required to be proved.}

By (1.5), our problem is invariant under translation and rotation transforms. Using the $k$-strictly  convexity assumption on the $\Om$,  we may  assume that
$$x_{0}=\textbf{0}  \ \   \text{and} \ \   \Om\subseteq\{x\in R^{n}|x_{k+1}>\eta_{1}|x_{1}|^{a_{1}}+...+\eta_{n-1}|x_{k}|^{a_{n-1}}\}$$
for some positive integer $k\leq n-2$ and positive constants $\eta_{1} , \cdots , \eta_{k}.$  It is enough to prove
that there is a positive constant $C$, depending only on  $a_{1},..., a_{k}, n, A, \alpha, \beta, \gamma$ and the
$\eta_1, \cdots, \eta_{k}$,  such that  for all $ y=(0,..., 0, y_{k+1}, 0, \cdots, 0)\in\Omega$,
\begin{equation}\label {3.1}
 Cy_{k+1}^{\mu}\geq |u(0,.., 0, y_{k+1}, 0, \cdots, 0)| .
 \end{equation}
 Obviously, (3.1) is more than the desired (1.6). In fact, if $x=(x_1, \cdots, x_k, x_{k+1}, \cdots, x_n)\in \Om$ such that $d_x=|x-x_0|=|x|$,
 but $x_j\neq 0$ for some $j\{1, \cdots, k, k+2, \cdots, n\}$, then by the convexity we have
 $$d_x\leq dist (x, \{x_{k+1}=0\})=x_{k+1}<|x|,$$
 which is impossible.

 Observing that for each $ 1\leq i \leq k$,
  \begin{equation}\label{3.2}
\begin{split}
&\Om\subseteq\{x\in R^{n}|\;x_{k+1}>\eta_{1}|x_{1}|^{a_{1}}+...+\eta_{k}|x_{k}|^{a_{k}}\}\\
&\subseteq\{x\in R^{n}|\;x_{k+1}>\eta_{i}|x_{i}|^{a_{i}}\}\\
&=\{x\in R^{n}| \; (\frac{\varepsilon}{\eta_{i}})^{\frac{2}{a_{i}}}(\frac{x_{k+1}}{\varepsilon})^{\frac{2}{a_{i}}}>|x_{i}|^{2}\} \\
&\subset \{x\in R^{n}|\;\delta(\varepsilon)(\frac{x_{k+1}}{\varepsilon})^{\frac{2}{a_{i}}}>|x_{i}|^{2}\},
\end{split}
\end{equation}
where
\begin{equation} \label{3.3}
 \delta(\varepsilon):=\max\limits_{1\leq i\leq k}(\frac{\varepsilon}{\eta_{i}})^{\frac{2}{a_i}},
\end{equation} ,
$0<\varepsilon <\min\{ 1, d, \min\limits_{1\leq i\leq k}\eta_i\} $ is to be determined, and    $d=diam(\Om)$ is the diameter of the $\Omega$.

{\bf Step 2.} \; {\sl Construct the function $H(x)=H(x_1, \cdots , x_{k+1})$ and calculate $D^2 H$.}

Again for each $i\in \{1, 2, ..., k\}$, let
\begin{equation} \label{3.4}
b_{i}= \frac{2}{a_{i}}\cdot \frac{ n+\alpha-\gamma}{ \bar a+1 +\beta-n-\gamma},
\end{equation}  and
\begin{equation}\label{3.5}
   H_{i}(x)=-[(\frac{x_{k+1}}{\varepsilon})^{\frac{2}{a_{i}}}-x_{i}^{2}]^{\frac{1}{b_{i}}}\; (1\leq i\leq k), \ \ H(x)=
  \sum\limits_{i=1}^{k}H_{i}(x),\ \ \ \forall x\in\Om .
\end{equation}
By a direct  calculation we have
\begin{equation}\label{3.6}
\begin{split}
H_{x_{i}x_{j}}&=H_{x_{j}x_{i}}=0, \ \ \  \text{ for all }\; i\; \text{ and }\; j\in \{k+2,  ..., n\}\\
H_{x_{i}x_{j}}&=0, \ \ \  \forall i, j \in \{1, 2, ..., k\},\ i\neq j\\
H_{x_{j}x_{j}}&=
 \frac{2}{b_{j}}|H_{j}|^{1-b_{j}}+\frac{4(b_{j}-1)}{b_{j}^{2}}|H_{j}|^{1-2b_{j}} x_{j}^{2}, \ \ \forall j\in \{1, 2, ..., k\}\\
 H_{x_{j}x_{k+1}}&=(H_{j})_{x_{j}x_{k+1}}=-\frac{4(b_{j}-1)}{a_{j}b_{j}^{2}}|H_{j}|^{1-2b_{j}}(\frac{x_{k+1}}{\varepsilon})^{\frac{2}{a_{j}}-1} x_{j}\frac{1}{\varepsilon}, \ \ \forall j\in \{1, 2, ..., k\} \\
H_{x_{k+1}x_{k+1}}&=\sum\limits_{i=1}^{k}[\frac{2(a_{i}-2)}{a_{i}^{2}b_{i}}|H_{i}|^{1-b_{i}}(\frac{x_{k+1}}{\varepsilon})^{\frac{2}{a_{i}}-2}
(\frac{1}{\varepsilon})^{2}\\
&\ \ +\frac{4(b_{i}-1)}{a_{i}^{2}b_{i}^{2}}|H_{i}|^{1-2b_{i}}(\frac{x_{k+1}}{\varepsilon})^{\frac{4}{a_{i}}-2}(\frac{1}{\varepsilon})^{2}].
 \end{split}
\end{equation}
It follows from (1.7), (3.4)  and (1.4) that
\begin{equation}\label {3.7}
 \mu =\frac{2}{a_{j}b_{j}} , \ \ \forall j\in \{1, 2, ..., k\}   \ \ \text{and}\;\; \mu \in (0, 1).
 \end{equation}
 Hence, for each $j\in \{1, 2, ..., k\}$,  by   (3.2) and (3.5) we obtain
\begin{equation}\label{3.8}
  (1-\delta(\varepsilon))^{\frac{1}{b_{k}}}\cdot(\frac{x_{k+1}}{\varepsilon})^{\mu}\leq
|H_{j}| \leq  (\frac{x_{k+1}}{\varepsilon})^{\mu}.
\end{equation}
  This, together with (3.2), (3.6) and the fact $b_j>0$, implies that  for each $j\leq k$ and  for all $x\in \Om$,
\begin{equation} \label{3.9}
\begin{split}
H_{x_jx_j}\geq &\frac{2}{b_j}\cdot \min\{(1-\delta(\varepsilon))^{\frac{1-b_j}{b_j}},  \; 1 \}\cdot(\frac{x_{k+1}}{\varepsilon})^{\mu (1-b_j)}\\
&-\frac{4|b_j-1|}{b_j^2}\cdot \max\{(1-\delta(\varepsilon))^{\frac{1-2b_j}{b_j}},  \; 1 \}\cdot(\frac{x_{k+1}}{\varepsilon})^{\mu(1-2b_j)}\cdot\delta(\varepsilon)\cdot(\frac{x_{k+1}}{\varepsilon})^{\mu b_j}\\
=&\frac{2}{b_j} [\min\{(1-\delta(\varepsilon))^{\frac{1-b_j}{b_j}}, \; 1\} -\delta(\varepsilon) \frac{2|b_j-1|}{b_j} \max\{(1-\delta(\varepsilon))^{\frac{1-2b_j}{b_j}},  \; 1\}]\\
& \ \ \cdot(\frac{x_{k+1}}{\varepsilon})^{\mu (1-b_{j})}\\
 := & c_{j}(\varepsilon)\cdot(\frac{x_{k+1}}{\varepsilon})^{\mu(1-b_j)}.
\end{split}
\end{equation}
Observing that  $\lim\limits_{\varepsilon\rightarrow 0} \delta(\varepsilon)=0$ by (3.3), we have
\begin{equation} \label{3.10}
 \lim\limits_{\varepsilon\rightarrow 0} c_{j}(\varepsilon)
=\frac{2}{b_{j}}=a_{j}\mu>0 , \ \ j=1, 2, \cdots, k,.
\end{equation}
Denote $ \xi_{j}(x)= |H_{j}(x)| (\frac{x_{k+1}}{\varepsilon})^{-\mu}$. It follows from $(3.8)$ that
\begin{equation}\label{3.11}
  (1-\delta(\varepsilon))^{\frac{1}{b_{j}}}\leq  \xi_{j}(x)\leq 1 , \ \    \forall x\in \Om, \ \ \forall j\in \{1, 2, ..., n-1\},
 \end{equation}
 and from $(3.6)$ and (3.7) that
\begin{equation}\label{3.12}
\begin{split}
H_{x_{k+1}x_{k+1}}&=\sum\limits_{j=1}^{k}[\frac{2(a_{j}-2)}{a_{j}^{2}b_{j}}\xi_{j}^{1-b_{j}}(\frac{x_{k+1}}{\varepsilon})^{\mu-2}
(\frac{1}{\varepsilon})^{2}
+\frac{4(b_{j}-1)}{a_{j}^{2}b_{j}^{2}}\xi_{j}^{1-2b_{j}}(\frac{x_{k+1}}{\varepsilon})^{\mu-2}(\frac{1}{\varepsilon})^{2}]\\
&=(\frac{1}{\varepsilon})^{2}\cdot\mu^{2}\sum\limits_{j=1}^{k}[(\frac{1}{\mu}-b_{j})\xi_{j}^{1-b_{j}}
+(b_{j}-1)\xi_{j}^{1-2b_{j}}]\cdot(\frac{x_{k+1}}{\varepsilon})^{\mu-2}\\
&:=(\frac{1}{\varepsilon})^{2}\cdot c_{k+1}(\varepsilon)\cdot(\frac{x_{k+1}}{\varepsilon})^{\mu-2}.
\end{split}
\end{equation}

It follows from (3.11) and (3.3) that \begin{equation}\label{3.13}
\lim\limits_{\varepsilon\rightarrow 0} \xi_{j}=1\ \ \text{ uniformly for }\; x\in \Om ,
  \end{equation} which yields
\begin{equation}\label{3.14}
 \lim\limits_{\varepsilon\rightarrow 0} c_{k+1}(\varepsilon)
 =k(\frac{1}{\mu}-1)\mu^{2}>0  \ \text{ uniformly for }\; x\in \Om .
 \end{equation}

 Again by $(3.2)$ and  $(3.6)$ we have
\begin{equation} \label{3.15}
\begin{split}
|H_{x_jx_{k+1}}|\leq&\frac{4|b_{j}-1|}{a_{j}b_{j}^{2}}\cdot\xi_{j}^{1-2b_{j}}\cdot(\frac{x_{k+1}}{\varepsilon})^{\mu-\frac{2}{a_{j}}-1}\cdot (\delta(\varepsilon))^{\frac{1}{2}}\cdot(\frac{x_{k+1}}{\varepsilon})^{\frac{1}{a_{j}}}\cdot\frac{1}{\varepsilon}\\
=&(\delta(\varepsilon))^{\frac{1}{2}}\cdot\frac{1}{\varepsilon}\cdot\frac{4|b_{j}-1|}{a_{j}b_{j}^{2}}\xi_{j}^{1-2b_{j}}\cdot(\frac{x_{k+1}}
{\varepsilon})^{\mu-\frac{1}{a_{j}}-1}\\
:=&(\delta(\varepsilon))^{\frac{1}{2}}\cdot\frac{1}{\varepsilon}\cdot \widetilde{c_{j}}(\varepsilon)\cdot(\frac{x_{k+1}}{\varepsilon})^{\mu-\frac{1}{a_{j}}-1}, \ \ \ \forall j\leq k,
\end{split}
\end{equation}
 where $\widetilde{c_{k}}(\varepsilon)$ satisfies
\begin{equation}\label{3.16}
 \lim\limits_{\varepsilon\rightarrow 0} \widetilde{c_{j}}(\varepsilon)
=\frac{4|b_{j}-1|}{a_{j}b_{j}^{2}}  \ \text{ uniformly for }\; x\in \Om
 \end{equation}
  by (3.13).

 Now we use (3.6) to write
 \begin{equation}\label{3.17}
  D^{2}H=\begin{pmatrix} E_{k+1} & 0 \\ 0 &   0_{n-k-1} \end{pmatrix},
  \end{equation}
  where $0_{n-k-1}$ is the zero square   matrix of order $n-k-1$,
  $$E_{k+1}:=\begin{pmatrix} A_{k} & \overrightarrow{v} \\ \overrightarrow{v}^{T} &  H_{x_{k+1}x_{k+1}} \end{pmatrix},$$
 $A_{k}=diag(H_{x_{1}x_{1}},..., H_{x_{k}x_{k}})$ and $\overrightarrow{v}^{T}=(H_{x_{1}x_{k+1}},..., H_{x_{k}x_{k+1}})$.
Notice that
\begin{equation}\label{3.18}
\det E_{k+1} =\ det A_{k}\cdot(H_{x_{k+1}x_{k+1}}-\overrightarrow{v}^{T}A_{k}^{-1}\overrightarrow{v}).
 \end{equation}

It follows from (3.9), (3.15) and (3.7) that
\begin{equation}\label{3.19}
\begin{split}
\overrightarrow{v}^{T}A_{k}^{-1}\overrightarrow{v}=&\sum\limits_{i=1}^{k}\frac{[H_{x_{i}x_{k+1}}]^{2}}{H_{x_{i}x_{i}}}  \\
\leq &\sum\limits_{i=1}^{k}\frac{[(\delta(\varepsilon))^{\frac{1}{2}}\cdot\frac{1}{\varepsilon}\cdot \widetilde{c_{i}}(\varepsilon)\cdot(\frac{x_{k+1}}{\varepsilon})^{\mu-\frac{1}{a_{i}}-1}]^{2}}{c_{i}(\varepsilon)\cdot(\frac{x_{k+1}}
{\varepsilon})^{\mu(1-b_{i})}}  \\
=&\delta(\varepsilon)\cdot(\frac{1}{\varepsilon})^{2}\cdot \sum\limits_{i=1}^{k}\frac{ (\widetilde{c_{i}}(\varepsilon))^{2}}{c_{i}(\varepsilon)}\cdot(\frac{x_{k+1}}{\varepsilon})^{\mu-2}.
\end{split}
\end{equation} By this,  (3.9),(3.12), (3.18) and again (3.7),  we have
 \begin{equation}\label{3.20}
\begin{split}
& \det E_{k+1}\geq [\sum_{i=1}^k c_{i}(\varepsilon)\cdot(\frac{x_{k+1}}{\varepsilon})^{\mu(1-b_{i})}]
  \cdot(\frac{1}{\varepsilon})^{2}[ c_{k+1}(\varepsilon) -\delta(\varepsilon)\cdot \sum\limits_{i=1}^{k}\frac{ (\widetilde{c_{i}}(\varepsilon))^{2}}{c_{i}(\varepsilon)} ]\cdot(\frac{x_{k+1}}{\varepsilon})^{\mu-2}  \\
&=(\frac{1}{\varepsilon})^{2}[\sum_{i=1}^k c_{i}(\varepsilon)] [ c_{k+1}(\varepsilon)-\delta(\varepsilon)\cdot \sum\limits_{i=1}^{k}\frac{ (\widetilde{c_{i}}(\varepsilon))^{2}}{c_{i}(\varepsilon)}]\cdot(\frac{x_{k+1}}{\varepsilon})^{\mu(k+1-b_{1} -b_2\cdots -b_{k})-2} \\
&= (\frac{1}{\varepsilon})^{2}\cdot [\sum_{i=1}^kc_{i}(\varepsilon)]  [ c_{k+1}(\varepsilon)-\delta(\varepsilon) \sum\limits_{i=1}^{k}\frac{ (\widetilde{c_{i}}(\varepsilon))^{2}}{c_{i}(\varepsilon)}]\cdot(\frac{x_{k+1}}{\varepsilon})^{(k+1)\mu-\bar a-2}  \\
& :=(\frac{1}{\varepsilon})^{2}\cdot \tau_{1}(\varepsilon)\cdot(\frac{x_{k+1}}{\varepsilon})^{(k+1)\mu-\bar a -2},
\end{split}
\end{equation}
 where
\begin{equation} \label{3.21}
\begin{split}
&\lim\limits_{\varepsilon\rightarrow 0} \tau_{1}(\varepsilon)
  =\lim\limits_{\varepsilon\rightarrow 0} c_{1}(\varepsilon)...c_{k}(\varepsilon)[ c_{k+1}(\varepsilon)-\delta(\varepsilon) \sum\limits_{k=1}^{k}\frac{ (\widetilde{c_{k}}(\varepsilon))^{2}}{c_{k}(\varepsilon)}]\\
&\ \ =\lim\limits_{\varepsilon\rightarrow 0} c_{1}(\varepsilon)...\lim\limits_{\varepsilon\rightarrow 0} c_{k}(\varepsilon)[\lim\limits_{\varepsilon\rightarrow 0} c_{k+1}(\varepsilon)-0]\\
&\ \ =a_{1}...a_{k}\mu^{k}k(\frac{1}{\mu}-1)\mu^{2}\\
 &\ \ = ka_{1}...a_{k}\mu^{(k+1)}(1-\mu)>0  \ \ \text{ uniformly for }\; x\in \Om .
\end{split}
\end{equation}
Here we have used (3.10), (3.14), (3.16) and (3.3).

{\bf Step 3.} \; {\sl Construct a function $G(x)=G(x_{k+1}, \cdots, x_n)$   such that for  $\varepsilon$ small enough and $M$ large enough, the function $W(x):=M(H(x)+G(x))\in C^2(\Om)\cap C(\bar \Om)$ is a (strict) sub-solution to problem (1.2). That is
\begin{equation}\label{3.22}
  F[ W]:=[F(x, W, \nabla W)]^{-1}\det  D^2 W(x)> 1, \  \forall x\in\Om \ \ \text{and} \;  W|_{\pom}\leq0 .
\end{equation}  }
Let
\begin{equation} \label{3.23}
G(x):=-(\frac{x_{k+1}}{\varepsilon})^{\mu}\sqrt{\Lambda^2-\sum_{i=k+2}^n x_i^2} , \ \ \   \Lambda:=\sqrt{2d^2+1}.
\end{equation}
By a direct  calculation we have
\begin{equation}\label{3.24}
\begin{split}
 &G_{x_i}=0 \ \ \  \text{ for all }\; i\in \{1, 2, \cdots, k\}\\
&G_{x_{k+1}}=- \frac{\mu}{ \varepsilon}(\frac{x_{k+1}}{\varepsilon})^{\mu-1}\sqrt{\Lambda^2-\sum_{i=k+2}^n x_i^2}\\
&G_{x_j}=(\frac{x_{k+1}}{\varepsilon})^{\mu}\frac{x_j}{\sqrt{\Lambda^2-\sum_{i=k+2}^n x_i^2}} \ \ \  \text{ for all }\; j\in \{k+2, \cdots, n\}\\
&G_{x_ix_j}=G_{x_jx_i}=0 \ \ \  \text{ for all }\; i,\;   j\in \{1,  ..., k\}\\
 &G_{x_ix_j}=G_{x_jx_i}=(\frac{x_{k+1}}{\varepsilon})^{\mu}\frac{1}{\sqrt{\Lambda^2-\sum_{i=k+2}^n x_i^2}}[\delta_{ij}+\frac{x_ix_j}{\Lambda^2-\sum_{i=k+2}^n x_i^2} ]\\
 &\ \ \ \    \text{ for all }\; i, j\in \{k+2, \cdots, n\}\\
 & G_{x_{j}x_{k+1}} =G_{x_{k+1}x_{j}}=\frac{\mu} {\varepsilon }(\frac{x_{k+1}}{\varepsilon})^{\mu-1}\frac{x_j}{\sqrt{\Lambda^2-\sum_{i=k+2}^n x_i^2}}\ \ \text{ for all }\; j\in\; \{1, \cdots, k\}\\
   G_{x_{k+1}x_{k+1}}&=\frac{\mu(1-\mu)}{ \varepsilon^2} (\frac{x_{k+1}}{\varepsilon})^{\mu-2} \sqrt{\Lambda^2-\sum_{i=k+2}^n x_i^2}.
 \end{split}
\end{equation}
Hence
\begin{equation}\label{3.25}
  D^{2}G(x)=\begin{pmatrix} 0_{k} & \vec{P} \\  \vec{Q} &   E_{n-k-1}(x) \end{pmatrix},
  \end{equation}
  where $0_{k}$ is the zero square   matrix of order $k$, and
  \begin{equation*}
\begin{split}
  &\vec{P}^T=\frac{\mu}{\varepsilon}(\frac{x_{k+1}}{\varepsilon})^{\mu-1}\frac{1}{\sqrt{\Lambda^2-\sum_{i=k+2}^n x_i^2}}(x_1, \cdots, x_k)\\
  &\vec{Q}^T=\frac{\mu}{ \varepsilon } (\frac{x_{k+1}}{\varepsilon})^{1-\mu}(\frac{x_1}{\sqrt{\Lambda^2-\sum_{i=k+2}^n x_i^2}},
  \cdots, \frac{x_k}{\sqrt{\Lambda^2-\sum_{i=k+2}^n x_i^2}}, \frac{\mu-1}{ \varepsilon } \frac{\varepsilon}{x_{k+1}} )\\
  &E_{n-k-1}(x):= \left[ G_{x_ix_j}\right]_{k+2\leq i, j\leq n}.
  \end{split}
\end{equation*}

  Since all the eigenvalues of $E_{n-k-1}(x)$ are
  $$(\frac{x_{k+1}}{\varepsilon})^{\mu}\frac{1}{\sqrt{\Lambda^2-\sum_{i=k+2}^n x_i^2}}, \ \ \cdots, \ \  (\frac{x_{k+1}}{\varepsilon})^{\mu}\frac{1}{\sqrt{\Lambda^2-\sum_{i=k+2}^n x_i^2}}, \ \ (\frac{x_{k+1}}{\varepsilon})^{\mu}\frac{\Lambda^2}{(\Lambda^2-\sum_{i=k+2}^n x_i^2)^{\frac{3}{2}}},$$
  we have
\begin{equation}\label{3.26}
 \begin{split}
 \det E_{n-k-1}&=(\frac{x_{k+1}}{\varepsilon})^{\mu(n-k-1)}\frac{\Lambda^2}{(\Lambda^2-\sum_{i=k+2}^n x_i^2)^{\frac{n-k+1}{2}}}\\
 &\geq (\frac{x_{k+1}}{\varepsilon})^{\mu(n-k-1)}\Lambda^{k+1-n}.
 \end{split}
\end{equation}

Therefore, for any positive constant $M$,  the function $W(x):=M(H(x)+G(x))\in C^2(\Om)\cap C(\bar \Om)$ and satisfies
  $W|_{\pom}\leq0$.  Moreover, letting $B:=D^2G-diag( 0_{k}, 0, E_{n-k-1})$
  and using the fact $\mu\in (0, 1)$, we see that the matrix $B\geq 0$ (non-negative definite). Using the elementary inequality
  $$\det (A_1+A_2)\geq [(\det A_1)^{\frac{1}{n}}+ (\det A_2)^{\frac{1}{n}}]^n$$
  for any non-negative definite matrix $A_1,A_2$ of $n$-order, by (3.17), (3.20), (3.25) and (3.26) we have obtained
   \begin{equation}\label{3.27}
   \begin{split}
 \det D^2 W(x)&=M^n\det [diag(E_{k+1},  E_{n-k-1})+  B]\\
 & \geq M^n[(\det diag(E_{k+1},  E_{n-k-1}))^{\frac{1}{n}}+ (\det B)^{\frac{1}{n}}]^n\\
 &\geq M^n\det diag(E_{k+1},  E_{n-k-1})\\
 &\geq M^n(\frac{1}{\varepsilon})^{2}  \tau_{1}(\varepsilon) (\frac{x_{k+1}}{\varepsilon})^{(k+1)\mu-\bar a -2}
 \cdot
 (\frac{x_{k+1}}{\varepsilon})^{\mu(n-k-1)}\Lambda^{k+1-n}\\
 &= M^n\Lambda^{k+1-n}(\frac{1}{\varepsilon})^{2}  \tau_{1}(\varepsilon)(\frac{x_{k+1}}{\varepsilon})^{n\mu-\bar a -2},
   \ \ x\in \Om .
 \end{split}
   \end{equation}

  To prove the claim stated in the Step 3, it is sufficient to find   a small  $\varepsilon$   and  a large $M$  such that the
  function $W$ satisfies the first inequality in (3.22).  For this purpose we need to  estimate
  $|\nabla  W|$  and $W$ in $\Omega$.

  Observing that $\Lambda =\sqrt{2d^2+1} $ and $d=diam(\Om)$,
  we have
  \begin{equation}\label{3.28}
   (\frac{x_{k+1}}{\varepsilon})^{\mu}\leq  |G|\leq \Lambda (\frac{x_{k+1}}{\varepsilon})^{\mu}, \ \  x\in \Om ,
   \end{equation}
 and
 $$|\nabla G(x)|^2=(\frac{x_{k+1}}{\varepsilon})^{2(\mu-1)}[(\frac{x_{k+1}}{\varepsilon})^{2}\cdot\frac{\sum_{i=k+2}^n x_i^2}{\Lambda^2-\sum_{i=k+2}^nx_i^2}
 +(\frac{\mu}{\varepsilon})^2(\Lambda^2-\sum_{i=k+2}^nx_i^2)].$$
 Observing that
 $$(\frac{x_{k+1}}{\varepsilon})^{2}\cdot\frac{\sum_{i=k+2}^n x_i^2}{\Lambda^2-\sum_{i=k+2}^nx_i^2}\leq (\frac{d}{\varepsilon})^2
 \frac{d^2}{\Lambda^2-d^2}\leq (\frac{d}{\varepsilon})^2\; \text{and}\;  \Lambda^2-\sum_{i=k+2}^nx_i^2\geq d^2+1\geq 1 $$
 for all $x\in \Om$, by the fact $\mu, \varepsilon \in (0,1)$ we get
  \begin{equation}\label{3.29}
 (\frac{\mu}{\varepsilon})^2 (\frac{x_{k+1}}{\varepsilon})^{2(\mu-1)}\leq |\nabla G(x)|^2
 \leq  (\frac{2\Lambda}{\varepsilon})^2 (\frac{x_{k+1}}{\varepsilon})^{2(\mu-1)} , \ \  x\in \Om .
   \end{equation}

 On the other hand, recall the definition for $\xi_i$ before (3.11).  It follows from (3.5) and (3.7) that
\begin{equation*}
\begin{split}
|H_{x_{k+1}}|=&|\sum\limits_{i=1}^{k}-\frac{1}{b_{i}}\cdot[(\frac{x_{k+1}}{\varepsilon})^{\frac{2}{a_{i}}}-x_{i}^{2}]^{\frac{1}{b_{i}}
-1}\cdot\frac{2}{a_{i}}(\frac{x_{k+1}}{\varepsilon})^{\frac{2}{a_{i}}-1}\cdot\frac{1}{\varepsilon}|\\
=&\frac{1}{\varepsilon}\cdot\mu\sum\limits_{i=1}^{k}|H_{i}|^{1-b_{i}}\cdot(\frac{x_{k+1}}{\varepsilon})^{\frac{2}{a_{i}}-1}\\
=&\frac{1}{\varepsilon}\cdot\mu\sum\limits_{i=1}^{k}\xi_{i}^{1-b_{i}}\cdot(\frac{x_{k+1}}{\varepsilon})^{\mu-1}.
\end{split}
\end{equation*}
 Using (3.2) to estimate $|x_i|$,
  we obtain that for $i\in \{1, 2, ..., k\}$,
\begin{equation*}
\begin{split}
|H_{x_{i}}|=&|\frac{2}{b_{i}}\cdot[(\frac{x_{k+1}}{\varepsilon})^{\frac{2}{a_{i}}}-x_{i}^{2}]^{\frac{1}{b_{i}}-1}\cdot x_{i}|\\
=&\frac{2}{b_{i}}\cdot|H_{i}|^{1-b_{i}}\cdot |x_{i}|\\
\leq&\frac{2}{b_{i}}\cdot\xi_{i}^{1-b_{i}}(\frac{x_{k+1}}{\varepsilon})^{\mu(1-b_{i})}\cdot (\delta(\varepsilon))^{\frac{1}{2}}(\frac{x_{k+1}}{\varepsilon})^{\frac{1}{a_{i}}}\\
=&(\delta(\varepsilon))^{\frac{1}{2}}\frac{2}{b_{i}}\cdot\xi_{i}^{1-b_{i}}(\frac{x_{k+1}}{\varepsilon})^{\mu-\frac{1}{a_{i}}}.
\end{split}
\end{equation*}

 Hence, for $i\in \{1, 2, ..., k\}$  we have
\begin{equation*}
\begin{split}
\frac{|H_{x_i}|}{|H_{x_{k+1}}|}\leq & \frac{(\delta(\varepsilon))^{\frac{1}{2}}2(b_i)^{-1}\cdot\xi_{i}^{1-b_i}(\frac{x_{k+1}}{\varepsilon})^{\mu-\frac{1}{a_i}}}
{\frac{1}{\varepsilon}\cdot\mu \sum\limits_{j=1}^{k}\xi_{j}^{1-b_{j}}\cdot(\frac{x_{k+1}}{\varepsilon})^{\mu-1}}\\
=&\varepsilon\cdot(\delta(\varepsilon))^{\frac{1}{2}}\cdot a_i\cdot(\frac{x_{k+1}}{\varepsilon})^{1-\frac{1}{a_i}}\cdot\xi_{i}^{1-b_i}(\sum\limits_{j=1}^{k}\xi_{j}^{1-b_j})^{-1}\\
\leq&\varepsilon\cdot(\delta(\varepsilon))^{\frac{1}{2}}\cdot a_{i}\cdot(\frac{d}{\varepsilon})^{1-\frac{1}{a_{i}}}\cdot\xi_{i}^{1-b_{i}}(\sum\limits_{j=1}^{k}\xi_{j}^{1-b_{j}})^{-1} \\
\leq &\varepsilon\cdot(\delta(\varepsilon))^{\frac{1}{2}}\cdot \max\limits_{1\leq k\leq k}a_{k}\cdot(\frac{d}{\varepsilon})^{1-( \max\limits_{1\leq j\leq k}a_j)^{-1}}\cdot\xi_{i}^{1-b_{i}}(\sum\limits_{j=1}^{k}\xi_{j}^{1-b_{j}})^{-1}  \\
 =& \hat {a} \varepsilon^{    \check{a}} (\delta(\varepsilon))^{\frac{1}{2}}    d^{1- \check{ a }}\xi_{i}^{1-b_{i}}(\sum\limits_{j=1}^{k}\xi_{j}^{1-b_{j}})^{-1}
\end{split}
\end{equation*}
for $\hat a:=\max\limits_{1\leq j\leq k}a_j$ and $ \check{ a }:= (\max\limits_{1\leq j\leq k}a_j)^{-1}. $  

 This implies that
\begin{equation*}
\frac{\sum\limits_{j=1}^{k}|H_{x_j}|}{|H_{x_{k+1}}|}\leq  \hat a  \varepsilon^{    \check{a}}
 (\delta(\varepsilon))^{\frac{1}{2}}   d^{1-      \check{a}}
:=\tau_{2}(\varepsilon).
 \end{equation*}

Observe that
\begin{equation} \label{3.30}
\lim\limits_{\varepsilon\rightarrow 0}\tau_{2}(\varepsilon)=0
\end{equation} by (3.3).  In this way we have obtained that
\begin{equation}\label {3.31}
\begin{split}
|\nabla H|\in &[|H_{x_{k+1}}|,\ \ \ (1+\tau_{2}(\varepsilon))|H_{x_{k+1}}|]\\
=&[\frac{\mu}{\varepsilon}\cdot \sum\limits_{i=1}^{k}\xi_{i}^{1-b_{i}}\cdot(\frac{x_{k+1}}{\varepsilon})^{\mu-1},\ \ \ \frac{\mu}{\varepsilon}\cdot(1+\tau_{2}(\varepsilon)) \sum\limits_{i=1}^{k}\xi_{i}^{1-b_{i}}\cdot(\frac{x_{k+1}}{\varepsilon})^{\mu-1}].
\end{split}
\end{equation}

Since $\mu\in (0, 1)$, we have
\begin{equation*}
\begin{split}
|\nabla H|\geq& \frac{\mu}{\varepsilon}\cdot \sum\limits_{i=1}^{k}\xi_{k}^{1-b_{i}}\cdot(\frac{d}{\varepsilon})^{\mu-1}\\
=& \varepsilon^{-\mu}d^{\mu-1}\cdot\mu\sum\limits_{i=1}^{k}\xi_{i}^{1-b_{i}}.
\end{split}
\end{equation*}

Using this estimate and (3.13),  one can find a
  $ \varepsilon_1\in (0, 1)$ such that
 $|\nabla H(x) |\geq1$ for all $x\in \Om $ and all $  \varepsilon\in (0, \varepsilon_1)$.   Hence, we have
\begin{equation}\label{3.32}
 |\nabla H(x)|\leq |(1+|\nabla H(x)|^{2})^{\frac{1}{2}} \leq  \sqrt{2}|\nabla H(x)|, \ \ \forall x\in \Om ,
\end{equation}
which, together with  the elementary inequality $\sqrt{|a|+|b|}\leq \sqrt{|a|}+\sqrt{|b|}$, yields
\begin{equation}\label{3.33}
\begin{split}
M|\nabla H| \leq (1+|\nabla W|^2)^{\frac{1}{2}}&\leq [1+2M^2(|\nabla H|^2+|\nabla G|^2)]^{\frac{1}{2}}\\
&\leq (1+2M^2|\nabla H|^2 )^{\frac{1}{2}}+(2 M^2|\nabla G|^2)^{\frac{1}{2}}\\
&\leq\sqrt{2} M(\sqrt{2}|\nabla H| +|\nabla G|)\\
&\leq \sqrt{2}M[\sqrt{2}|\nabla H| +\frac{2\Lambda}{\varepsilon}(\frac{x_{k+1}}{\varepsilon})^{\mu-1}]\\
&\leq  M( 2+\frac{2\sqrt{2}\Lambda}{\mu\sum_{i=1}^k\xi_i^{1-b_i}})|\nabla H| ,
\end{split}
\end{equation}
where we have assumed $M\geq 1$ and used  (3.29) and (3.31) for the last two inequalities.

Next, denote $\bar \xi(x)=[\mu\sum\limits_{i=1}^{k}\xi_{i}^{1-b_{i}}]$.
  Whether  $\gamma\geq0$ or $\gamma<0$, combing (3.33) with (3.31) we always have
\begin{equation}\label{3.34}
\begin{split}
&(1+|\nabla  {W}|^{2})^{\frac{-\gamma}{2}}\geq M^{-\gamma} \min\{1, \; (2+\frac{2\sqrt{2}\Lambda}{\bar \xi})^{-\gamma}\} |\nabla H|^{-\gamma}\\
&\geq M^{-\gamma} (\frac{1}{\varepsilon})^{-\gamma}\cdot \min\{1, \; ( 2+ \frac{2\sqrt{2}\Lambda}{\bar \xi})^{-\gamma}\}
\min\{  1, \;  (1+\tau_{2}(\varepsilon))^{-\gamma}\}{\bar \xi }^{-\gamma} \cdot(\frac{x_{k+1}}{\varepsilon})^{\gamma-\mu\gamma}\\
 & :=M^{-\gamma}(\frac{1}{\varepsilon})^{-\gamma}\cdot\tau_{3}(\varepsilon)\cdot(\frac{x_{k+1}}{\varepsilon})^{\gamma-\mu\gamma}\ \ in \; \Omega .
\end{split}
\end{equation}

 Observe that $\tau_{3}(\varepsilon)$  satisfies
\begin{equation}\label{3.35}
 \lim\limits_{\varepsilon\rightarrow 0}\tau_{3}(\varepsilon)=\min\{1, \; ( 2+\frac{2\sqrt{2}\Lambda}{\mu k})^{ -\gamma }  \}  (\mu k)^{-\gamma}>0   \ \ \text{ uniformly for }\; x\in \Om
  \end{equation}
due to  (3.13) and  (3.30).

On the other hand, it is obvious from the expressions of $H$ and $G$ that
$$ (\frac{x_{k+1}}{\varepsilon})^{\mu}\leq |H+G|\leq (\Lambda+k)(\frac{x_{k+1}}{\varepsilon})^{\mu}, \ \ x\in \Om.$$
Therefore, no matter $\alpha\geq 0 $ or $\alpha <0$, we always  have
 \begin{equation}\label{3.36}
\begin{split}
|W|^{\alpha}&=M^{\alpha}|H+G|^{\alpha} \\
&\geq M^{\alpha}  \min\{1,\; (\Lambda+k)^{\alpha}\}   (\frac{x_{k+1}}{\varepsilon})^{\mu\alpha},\ \  \forall x\in \Om  .
\end{split}
\end{equation}
Recalling that $d_x=dist(x, \partial \Om)\leq x_{k+1}$ and  $\beta-(n+1)\geq0$, we have
\begin{equation}\label{3.37}
d_{x}^{n+1-\beta}\geq x_{k+1}^{n+1-\beta}=\varepsilon^{n+1-\beta}(\frac{x_{k+1}}{\varepsilon})^{n+1-\beta}, \ \  \forall x\in \Om .
\end{equation}
Finally, we use the construction condition
(1.4) and apply (3.34), (3.36) and (3.37) to obtain that
\begin{equation*}
\begin{split}
& [F(x,  W, \nabla  W)]^{-1}\geq \frac{1}{A }d_{x}^{n+1-\beta}| W|^{\alpha}\cdot(1+|\nabla W|^{2})^{\frac{-\gamma}{2}}\\
& \geq  \frac{1}{A }\varepsilon^{n+1-\beta}(\frac{x_{k+1}}{\varepsilon})^{n+1-\beta}\cdot M^{\alpha}    \min\{1,\; (\Lambda+k)^{\alpha}\}
 (\frac{x_{k+1}}{\varepsilon})^{\mu\alpha}\cdot M^{-\gamma}(\frac{1}{\varepsilon})^{-\gamma}\tau_{3}(\varepsilon) (\frac{x_{k+1}}{\varepsilon})^{\gamma-\mu\gamma}\\
&=A^{-1}M^{\alpha-\gamma}\varepsilon^{n+1-\beta+\gamma}\tau_{3}(\varepsilon)   \min\{1,\; (\Lambda+k)^{\alpha}\} (\frac{x_{k+1}}{\varepsilon})^{n+1-\beta+\gamma+\mu\alpha-\mu\gamma},
\end{split}
\end{equation*} which, together with   (3.27), yields
\begin{equation}\label{3.8}
\begin{split}
& F[W]:= \det D^2 W \cdot [F(x,  W, \nabla  W)]^{-1}    \\
&\ \ \geq M^n\Lambda^{k+1-n}(\frac{1}{\varepsilon})^{2}  \tau_{1}(\varepsilon)(\frac{x_{k+1}}{\varepsilon})^{n\mu-\bar a -2}\\
&\ \ \  \cdot A^{-1}M^{\alpha-\gamma}\varepsilon^{n+1-\beta+\gamma}\tau_{3}(\varepsilon)  \min\{1,\; (\Lambda+k)^{\alpha}\} (\frac{x_{k+1}}{\varepsilon})^{n+1-\beta+\gamma+\mu\alpha-\mu\gamma}\\
&\ \  = A^{-1}M^{n+\alpha-\gamma}\varepsilon^{n-1-\beta+\gamma}  \Lambda^{k+1-n}  \min\{1,\; (\Lambda+k)^{\alpha}\}   \tau_{1}(\varepsilon)\tau_{3}(\varepsilon)\\
 & \ \ \  \cdot (\frac{x_{k+1}}{\varepsilon})^{\mu(n+\alpha-\gamma)-\bar a-\beta +n-1+\gamma}\\
 &=A^{-1} M^{n+\alpha-\gamma}\cdot\varepsilon^{n-1-\beta+\gamma}  \Lambda^{k+1-n}  \min\{1,\; (\Lambda+k)^{\alpha}\}  \tau_{1}(\varepsilon)\tau_{3}(\varepsilon),
 \end{split}
\end{equation}
where we have used (1.7) for the last equality.
However,   by (3.21) and (3.35) one can find a positive constant   $\varepsilon=\varepsilon_{0}<\min\{\eta_{1},..., \eta_{n-1},  d, 1,  \varepsilon_1 \}$ and $\delta_0=C(\varepsilon_{0})\>0$ such that
\begin{equation*}
A^{-1} \varepsilon_0^{n-1-\beta+\gamma}  \Lambda^{k+1-n}  \min\{1,\; (\Lambda+k)^{\alpha}\}  \tau_{1}(\varepsilon_0)\tau_{3}(\varepsilon_0)
 >\delta_0, \ \ \forall x\in \Om .
\end{equation*}
Since (1,4) means that $n+\alpha-\gamma>0$,  we can take $M=M_{0}=C(\delta_{0}) $ such that
\begin{equation*}
A^{-1}M_{0}^{n+\alpha-\gamma}\cdot\varepsilon_0^{n-1-\beta+\gamma}  \Lambda^{k+1-n}   \min\{1,\; (\Lambda+k)^{\alpha}\}  \tau_{1}(\varepsilon_0)\tau_{3}(\varepsilon_0)>1, \ \ \forall x\in \Om ,
  \end{equation*}
which together with (3.38) implies
$$F[ W ] = [F(x, W, \nabla  W)]^{-1}\det D^{2} W> 1, \ \ \ \forall x\in\Omega.$$
 This proves the desired (3.22), and the proof of Step 3 is completed.
\vskip 0.5cm

{\bf Step 4.} \; {\sl Finish the proof of (3.1)}.

Now since $u\in C(\overline \Om) $ is convex with $u=0$ on $\partial \Om$,  then $u\leq 0$ in $\Om$.  Furthermore, since $u$ is a viscosity solution to
problem (1.2), by (3.22) and Lemma 2.1 we have $ W\leq u$   on $\bar \Omega$. This implies  $ |u|\leq | W|$   on $\bar \Omega$.
In particular, $\forall y=(0,...,0, y_{k+1},0, \cdots, 0 )\in\Omega$, we have
 $$
|u(0,.., 0, y_{k+1}  , 0, \cdots, 0)|\leq |W(0,.., 0, y_{k+1},0, \cdots, 0)|
=M_{0}(k+\Lambda)(\frac{1}{\varepsilon_{0}})^{\mu}\cdot y_{k+1}^{\mu}.$$
 In this way, we have proved (3.1) and so Theorem 1.1.

\vskip15pt

{\bf Proof of Corollary 1.2.}\;\;
  By lemma 2.3, it is sufficient to prove
\begin{equation}\label {3.39}
\begin{split}
  |u(y)|\leq C{d_{y}}^{\mu}, \ \  \forall y\in \Om
\end{split}
\end{equation}
For each  $y\in \Om$ we can find a $x_0\in \partial \Om$ such that $d_{y}=|y-x_0|$.
However, $x_0$ satisfies the $k$-strictly convex condition, because of the $k$-strictly convex assumption of the $om$.
Hence, (3.39)   follows directly from Theorem 1.1.

 \vskip15pt

{\bf Proof of Corollary 1.3.}\;\;
  This was proved for the classical solutions $u\in C^2(\Om)\cap C(\bar \Om)$ in Proposition 3.1 in \cite{[LL]}.
    Its argument is easily extended to
  the case of viscous solutions.

 Consider the function
 \begin{equation}\label{3.40}
  W(x):=-M(x_n)^{\mu_0}(N^2-\sum_{i=1}^{n-1}x_i^2), \ \ x=(x' , x_n)\in \Om.
\end{equation}
 Under the assumption of Corollary 1.3, it follows from the proof of Proposition 3.1 in \cite{[LL]} (or Theorem 1.1 in \cite{[JLT]}
 for the special case $\gamma=0$)
 that the $W$ satisfies (3.22)  for some positive constant    $N=C(n, d)$
 and $M=C(A,\alpha, \beta, n, \gamma, d)$.
   Repeating the arguments  of Step 4 in the proof  of Theorem 1.1 and the proof of Corollary 1.2, one can obtain the desired result
 of Corollary 1.3.

\end{document}